\documentclass[11pt]{article}
\usepackage{amsmath,amsthm, amssymb, amsfonts,latexsym,amscd}
\usepackage{graphicx}
\usepackage{tikz}
\usetikzlibrary{shadings,patterns}
\newtheorem{theorem}{Theorem}[section]
\newtheorem{corollary}[theorem]{Corollary}
\newtheorem{lemma}[theorem]{Lemma}
\newtheorem{proposition}[theorem]{Proposition}

\newtheorem{example}[theorem]{Example}

\begin{document}

\title{\bf Different central parts of trees and their pairwise distances}
\author{Dinesh Pandey\footnote{Supported by UGC Fellowship scheme (Sr. No. 2061641145), Government of India} \and Kamal Lochan Patra}
\date{}
\maketitle

\begin{abstract}
We determine the tree which maximizes the distance between characteristic set and subtree core over all trees on $n$ vertices. The asymptotic nature of this distance is  also discussed. The problem of extremizing the distance between different central parts of trees on $n$ vertices with fixed diameter is studied.\\



\end{abstract}

\section{Introduction} 

Throughout this paper, graphs are simple, finite and undirected. Let $G$ be a graph with vertex set $V(G)$ and edge set $E(G)$. A tree is a connected acyclic graph. A pendant vertex in a tree is a vertex having degree $1$. For a tree $T$ with $u,v \in V(T),$ the distance $d_T(u,v)$ or simply $d(u,v)$, is the number of edges in the path joining $u$ and $v$. We denote the distance between two subsets $U$ and $V$ of $V(T)$ by $d_T(U,V)$ or simply by $d(U,V)$  and define it as $d(U,V)= \min\{d(u,v): u\in U, v\in V\}$. For a vertex $v \in V(T)$, $e(v)=\max\{d(v,u)| u \in V(T)\}$ is called the {\it eccentricity} of $v$ in $T.$ A vertex of minimum eccentricity is called a central vertex of $T$ and the set of all central vertices is called the {\it center} of $T.$ We denote the center of a tree by $C(T)$. 

For $v\in V(T)$, a branch at $v$ is a maximal subtree of $T$ containing $v$ as a pendant vertex. The {\it weight} of $v$ is the maximal number of edges in any branch of $T$ at $v.$ A vertex of minimal weight is called a centroid vertex of $T$ and the set of all centroid vertices is called the {\it centroid} of $T.$ We denote the centroid of $T$ by $C_d(T)$. The following result is due to Jordan. 
\begin{proposition}\label{prop:cnt}(\cite{H},Theorem 4.2,Theorem 4.3)
\begin{enumerate}
\item	The center of a tree  consists of either a single vertex or two adjacent vertices.
\item The centroid of a tree consists of either a single vertex or two adjacent vertices.
\end{enumerate}
\end{proposition}

 For $v \in V(T),$ the distance of $v$ in $T$, denoted by $g(v),$ is defined as $g(v)=\sum_{u \in V(T)}d(u,v)$. A vertex of minimum distance is called a median vertex of $T$ and the set of all median vertices is called the {\it median} of $T.$  In \cite{Z}, Zelinka proved the following facts regarding median.
\begin{proposition}\label{prop:med1}(\cite{Z},Theorem 2 and 3)
The median of a tree consists of either a single vertex or two adjacent vertices and it coincides with the centroid.
\end{proposition}

By assuming the vertices of $T$ as telephone lines and the path between two vertices $u$ and $v$ as a representaion of a telephone call between $u$ and $v$, Mitchell in \cite{M} defined another central part of a tree called the telephone center. Assuming that at a given time a vertex can be involved in only one call, define the {\it switchboard number} of $v$ denoted by $sb(V)$ as the maximum number of calls which can pass through $v$ at a given time. The {\it telephone center} of $T$ is the set of vertices having maximum switchboard number. The following result regarding  the telephone center is due to Mitchell.
\begin{proposition}\label{prop:tel}(\cite{M},Corollary 3)
The telephone center of a tree consists of either a single vertex or two adjacent vertices and it coincides with the centroid.
\end{proposition}

In \cite{Sw}, Szekely and Wang  defined another central part of a tree as following: For $v\in V(T),$ let $f_T(v)$ be the number of subtrees of $T$ containing $v.$ The {\it subtree core} of $T$ is the set of vertices $v$ for which $f_T(v)$ is maximum. We denote the subtree core of $T$ by $S_c(T)$.
\begin{proposition}\label{SC-00}(\cite{Sw},Theorem 9.1)
The subtree core of a tree is either a single vertex or two adjacent vertices.
\end{proposition}
 One can find a tree (see Example \ref{ex}) in which center, centroid and subtree core are pairwise disjoint. Since the median and the telephone center coincide with the centroid, we mainly have three distinct central parts of a tree defined on the basis of edges and distances. All the above central parts are combinatorially defined. We now algebraically define a central part of a tree which is different from all the above combinatorially defined centres.
 
For a graph $G$ with $V(G)=\{v_1,v_2,\ldots,v_n\}$, the degree matrix $D(G)=(d_{ij})$ is the $n \times n$ diagonal matrix with $d_{ii}$ is equal to degree of the vertex $v_i$ for $i=1,2,\ldots, n.$ The adjacency matrix $A(G)=(a_{ij})$ is the $n\times n$ matrix where $a_{ij}=1$ if $v_i$ and $v_j$ are adjacent and $0$ otherwise. The Laplacian matrix of $G$, denoted by $L(G)$, is defined as $L(G)=D(G)-A(G)$. The Laplacian matrix $L(G)$ is a symmetric, positive semi definite matrix. The smallest eigenvalue of $L(G)$ is $0$ with all one vector as an eigenvector. The second smallest eigenvalue of $L(G)$ is called algebraic connectivity of $G$ as it is positive if and only if $G$ is connected (see \cite{F1}). We denote the second smallest eigenvalue of $L(G)$ by $\mu$. An eigenvector corresponding to $\mu$ is called a Fiedler vector of $G.$ If $Y$ is a Fiedler vector of $G$, by $Y(v)$ we mean the co-ordinate of $Y$ corresponding to the vertex $v.$ A vertex $v\in V(G)$ is called a characteristic vertex  if $Y(v)=0$ and there exists a vertex $u\in V(G)$ adjacent to $v$ such that $Y(u)\neq 0.$ An edge $e=\{u,v\}$ is called a characteristic edge of $G$ if $Y(u)Y(v)<0.$ The characteristic set of $G$ with respect to  $Y$ is the set of all characteristic vertices and characteristic edges of $G$. The concept of characteristic set in terms of characteristic vertices and characteistic edges was first introduced by Bapat and Pati in \cite{Bp}. The following results shows the importance of the study of characteristic set of a tree.
 
\begin{proposition} (\cite{F2},Theorem 3,14 and \cite{Rm},Theorem 2) 
The characteristic set of a tree is either a vertex or an edge which is same for any Fiedler vector.
\end{proposition}
 We can consider the characteristic edge as two adjacent vertices and hence for a tree $T$ the characteristic set behaves like a centre. Since the characteristic set of a tree $T$ is independent of the choice of Fiedler vector, we denote it by $\chi(T)$. Next we have given one example of a tree where center, centroid , subtree core and characteristic set are disjoint.
  \begin{example}\label{ex}
In the tree $T$  (See Figure \ref{fgr-0}), the vertex $6$ is the center as its eccentricity is $5$, less than any other vertex. The vertex $9$ is the centroid as it has weight $8$,  less than any other vertex and the subtree core is the vertex $10$, as the number of subtrees containing $10$ is $10\times 2^7$, more than any other vertex. Also $\mu(T)=.0483$ and $Y=(-0.4116,-0.3917,-0.3528,-0.2970,\\-0.2267,-0.1455,
	-0.0573,0.0337,0.1231,0.2065,0.2170,0.2170,0.2170,0.2170,0.2170,\\0.2170,0.2170)^T$ is a Fiedler vector. So $\chi(T)=\{7,8\},$ which is disjoint from each of the center, centroid and subtree core.
 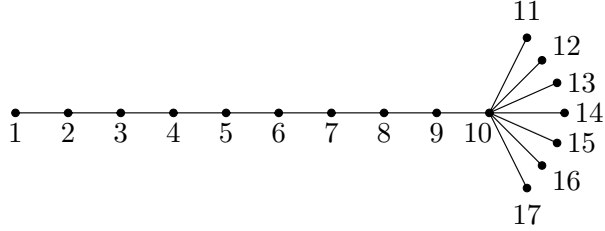
\begin{figure}[h] 
 	\begin{center}
 		\begin{tikzpicture}
 		\filldraw (0,0)node [below]{1} circle [radius= .5mm]--(.7,0)node [below]{2} circle[radius=.5mm]--(1.4,0)node [below]{3} circle [radius= .5mm]--(2.1,0)node [below]{4} circle[radius=.5 mm]--(2.8,0)node [below]{5} circle [radius= .5mm]--(3.5,0)node [below]{6} circle[radius=.5 mm]--(4.2,0) node [below]{7}circle [radius= .5mm]--(4.9,0)node [below]{8} circle[radius=.5 mm]--(5.6,0) node [below]{9}circle [radius= .5mm]--(6.3,0)circle[radius=.5 mm]--(7.3,0)node [right]{14} circle[radius=.5 mm];
 			\draw (6.15,0) node [below]{10};
 			\filldraw (6.8,1)node[above,outer sep=3pt]{11}circle [radius=.5mm]--(6.3,0);
 			\filldraw (7,.7) circle [radius=.5mm]--(6.3,0);
 			\filldraw (7.2,.4) node[right]{13}circle [radius=.5mm]--(6.3,0);
 			\draw (7,.9)node[right]{12};
 			\filldraw (6.8,-1)node[below,outer sep=3pt]{17}circle [radius=.5mm]--(6.3,0);
 			\filldraw (7,-.7) circle [radius=.5mm]--(6.3,0);
 			\filldraw (7.2,-.4) node[right]{15}circle [radius=.5mm]--(6.3,0);
 			\draw (7,-.9)node[right]{16};
 			\end{tikzpicture}
 		\end{center}
 		\caption{ Tree with disjoint central parts}\label{fgr-0}
 	\end{figure}
 	
 \end{example}

Thus we have four central parts in a tree which may be pairwise disjoint for some trees. It is natural to ask how far or close any two of these centres are in a tree on $n$ vertices.  All these four centres coincide in both path and star. So over trees on $n$ vertices, the minimum distance between any two of these centres is zero. Maximizing the distance between any two of these centres over trees on $n$ vertices are studied by many researchers in last two decades. For the pair \{center , centroid\}, the distance is studied in \cite{P} and \cite{Sswy}. Maximum distances for the pair \{center, characteristic set\} and \{centroid, characeristic set\} are studied in \cite{P} and \cite{Afjk}. Also  in \cite{Dp} and \cite{Sswy}, maximum distances for the pair \{center, subtree core\} and \{centroid, subtree core\} are studied. In this paper, we have obtained a tree which maximizes  the distance between subtree core and characteristic set over all trees on $n$ vertices. We have also studied the pairwise distance of these centres over trees on $n$ vertices with fixed diameter $d.$

The paper is organized in the following way: In Section $2$, we  discuss some results related to characteristic set and subtree core of trees which are very important and useful to prove our main results. In Section $3$, we obtain the tree which maximize the distance  between characteristic set and subtree core over all trees on $n$ vertices. We also study the asymptotic nature of this distance. In Section $4$, we discuss the  problem of extremizing the  distance between any two central parts of  trees on $n$ vertices with fixed diameter.  We partially answer some of these problems.

\section{Preliminaries}

In this section we will discuss some results related to our problem of maximizing the distance between characteristic set and subtree core over all trees on $n$ vertices. Following three lemmas are related to subtree core of trees and important for our study. 

\begin{lemma}\label{SC-0}(\cite{Sw},Theorem 9.1)
Let $T$ be a tree and $u,v,w\in V(T).$ If $\{u,v\},\{v,w\}\in E(T)$ then $2f_T(v)>f_T(u)+f_T(w).$
\end{lemma}
 
\begin{lemma}\label{SC-01}(\cite{Dp},Lemma 2.2)
Let $T$ be a tree, $v \in S_c(T)$ and $y$ be a pendant vertex of $T$ not adjacent to $v.$ If $\tilde{T}$ is the tree obtained from $T$ by detaching $y$ from $T$ and adding it as a pendant vertex adjacent to $v,$ then $S_c(\tilde{T})=\{v\}.$
\end{lemma}

\begin{lemma}\label{SC-02} (\cite{Dp},Lemma 3.1)
Let $T$ be a tree, $v \in S_c(T)$ and $B$ be a branch at $v$. Let $u$ be the vertex in $B$ adjacent to $v$ and $x$ be a pendant vertex of $T$ in $B$. Suppose that $B$ is not a path. Let $y$ be the vertex closest to $x$ with $d(y)\geq 3$ and $[y,y_1,y_2,\ldots,y_m=x]$ be the path connecting $y$ and $x.$ Let $z \neq y$ be a vertex of $B$ such that the path from $v$ to $z$ contains $y$ but not $y_1.$ Let $\tilde{T}$ be the tree obtained from $T$ by detaching the path $[y_1,y_2,\ldots,y_m]$ from $y$ and attaching it to $z.$ Then $f_{\tilde{T}}(v)>f_{\tilde{T}}(u).$ 
 \end{lemma}

\begin{figure}[!h]
\begin{center}
\begin{tikzpicture}[scale =.5]
\filldraw (0,0)node[above,outer sep=5pt]{1} circle [radius=1mm];
\filldraw (2,0)node[above,outer sep=5pt]{2} circle [radius=1mm];
\filldraw (4,0)node[above,outer sep=5pt]{3} circle [radius=1mm];
\filldraw (7,0)node[above,outer sep=5pt]{n-g-2} circle [radius=1mm];
\filldraw (9,0)node[above,outer sep=5pt]{n-g-1} circle [radius=1mm];
\filldraw (11,0)node[right,outer sep=5pt]{n-g} circle [radius=1mm];
\filldraw (9.6,2)node[above,outer sep=0pt]{n-g+1} circle [radius=1mm];
\filldraw (11,2.6)node[above,outer sep=0pt]{n-g+2} circle [radius=1mm];
\filldraw (9.6,-2)node[below,outer sep=0pt]{n} circle [radius=1mm];
\filldraw (11,-2.6)node[below,outer sep=0pt]{n-1} circle [radius=1mm];
\draw(0,0)--(2,0)--(4,0);
\draw (7,0)--(9,0)--(11,0);
\draw (9.6,2)--(11,0)--(11,2.5);
\draw (9.6,-2)--(11,0)--(11,-2.5);
\draw [dash pattern= on 2pt off 2pt](4,0)--(7,0);
\clip (11,-2.6) rectangle (14,3);
\draw [dash pattern=on 2pt off 2pt] (11,0) circle [radius=2.6cm];
\end{tikzpicture}
\end{center}
\caption{Path-star tree}\label{fig:2}
\end{figure}
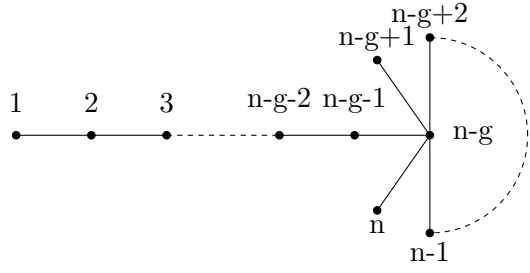

 A path-star tree  $P_{n-g,g}$ is the tree obtained by identifying the center of the star $K_{1,g}$ with a pendant vertex of the path $P_{n-g},\; 2\leq g\leq n-3$ (see Figure \ref{fig:2}). For more details on path-star trees, we refer \cite{Dp,P}. Path-star tree plays an important role in maximizing the pairwise distance of different centres of trees. The subtree core of  path-star trees are discussed in the next lemma.
 
 \begin{lemma}(\cite{Dp},Theorem 2.4)\label{SC-03}
The subtree core of the path-star tree $P_{n-g,g}$ is given by
\begin{equation*}
S_c(P_{n-g,g})=\begin{cases}
 \begin{cases}
 \{\frac{n-g+2^g}{2}\}, &\text{if $n-g$ is even}\\
 \{\frac{n-g-1+2^g}{2},\frac{n-g+1+2^g}{2}\}, &\text{if $n-g$ is odd.}
\end{cases}, & \textit{if $2^g+1 \leq n-g$},\\
\{n-g\}, & \textit{if $2^g+1 > n-g$}.
\end{cases}
\end{equation*}
 \end{lemma}
 
We now discuss some important results related to the study of the position of characteristic set in a tree. Let $v$ be a vertex of a tree $T.$ Let $T_1,T_2,\cdots, T_k$ be the connected components of $T - v$. For each such  component, let $\hat{L}(T_i), i=1,2,\cdots, k$ denote the principal submatrix of the Laplacian matrix $L$ corresponding to the vertices of $T_i$. Then  $\hat{L}(T_i) $ is invertible and $\hat{L}(T_i)^{-1} $ is a positive matrix which is called the bottleneck matrix for $T_i.$

By Perron-Frobenius Theorem, $\hat{L}(T_i)^{-1} $ has a simple dominant eigenvalue, called Perron value of $T_i$ at $v$. The component $T_j$ is called a Perron component at $v$ if its Perron value is maximal among the components $T_1, T_2,\cdots, T_k$, at $v$. The next result  describing the entries of  bottleneck matrices is very useful.

 \begin{lemma} (\cite{Kns},Proposition 1)\label{L-01}
 Let $T$ be a tree and let $v \in V(T).$ Let $T_1$ be a component of $T - v$ and  $L_1$ be the submatrix of $L(T)$ corresponding to $T_1.$ Then $L_1^{-1}=(m_{ij})$, where $m_{ij}$ is the number of edges in common between the paths $P_{iv}$ and $P_{jv},$ where $P_{iv}$ denotes the path joining $i$ and $v$.
\end{lemma}

A connection between Perron components and characteristic set of a tree  is described in next three results.
\begin{theorem}(\cite{Kns},Corollary 1.1)\label{pt1}
Let $T$ be a tree on $n$ vertices. Then the edge $\{i,j\}$ is the characteristic edge of $T$ if and only if the component $T_i$ at vertex $j$ containing the vertex $i$ is the unique Perron  component at $j$ while the component $T_j$ at vertex $i$ containing the vertex $j$ is the unique Perron component at $i.$
\end{theorem}

\begin{theorem}(\cite{Kns},Corollary 2.1)\label{pt2}
 Let $T$ be a tree on $n$ vertices. Then the vertex $v$ is the characteristic vertex of $T$ if and only if there are two or more Perron components of $T$ at $v$.
\end{theorem}

\begin{theorem} (\cite{Kns},Proposition 2)\label{pt3} 
Let $T$ be a tree. Then for any vertex $v$ that is neither a characteristic vertex nor an end vertex of the characteristic edge, the unique Perron component at $v$ contains the characteristic set of $T.$
\end{theorem}

The following two results are related to the study of the position of characteristic set in some path-star trees and are useful to prove our main result.
 
 \begin{lemma}\label{CS-01}(\cite{P},Lemma 2.2)
  The  characteristic set of $P_{n-2,2}$ is given by
 \begin{equation*}
 \chi(P_{n-2,2})=\begin{cases}
 \{\frac{n}{2},\frac{n}{2}+1\}, &\textit{if $n$ is even},\\
  \{\frac{n-1}{2},\frac{n+1}{2}\}, &\textit{if $n$ is odd}.
 \end{cases}
 \end{equation*}
\end{lemma}

 \begin{lemma} (\cite{P},Proposition 3.1, 3.2, 3.3 and 3.4)\label{CS-02}
 	\begin{itemize}
 		\item[1.] Let $\{i,i+1\}$ be the characteristic edge of $P_{n-g,g}$ where $g\geq 3$ and $2 \leq i\leq n-g-1.$ Then the characteristic set of $P_{n-g+1,g-1}$ lies between the vertices $i$ and $i+2$ with neither $i$ nor $i+2$ as a characteristic vertex.
 		
 		\item[2.] Let $i$ be  the characteristic vertex of $P_{n-g,g},$ where $g\geq3.$ Then $\{i,i+1\}$ is the characteristic edge of $P_{n-g+1,g-1}.$
 		
 		\item [3.] Let $\{i,i+1\}$ be the characteristic edge of $P_{n-g,g}$ where $g\leq n-4$. Then the characteristic set of $P_{n-g-1,g+1}$ lies between the vertices $i-1$ and $i+1$ with neither $i-1$ nor $i+1$ as a characteristic vertex.
 
  		 \item [4.] Let $i$ be  the characteristic vertex of $P_{n-g,g},$ where $g\leq n-4.$ Then $\{i-1,i\}$ is the characteristic edge of $P_{n-g-1,g+1}.$
 
\end{itemize}

 \end{lemma}
 \begin{lemma} \label{L-03}
For the path-star tree $P_{n-2,2}$, $d_{P_{n-2,2}}(S_c,\chi)=0.$
\end{lemma}
\begin{proof}
By  Lemma \ref{SC-03}, if $n<7$ then $S_c(P_{n-2,2})=\{n-2\}$. For $n\geq 7$, 
\begin{equation*}
 S_c(P_{n-2,2})=\begin{cases}
 \{\frac{n}{2}+1\}, &\textit{if $n$ is even},\\
  \{\frac{n+1}{2},\frac{n+3}{2}\}, &\textit{if $n$ is odd}.
 \end{cases}
 \end{equation*}
 
 Hence by Lemma \ref{CS-01}, $d_{P_{n-2,2}}(S_c,\chi)=0.$
\end{proof}

 \section{Distance between characteristic set and the subtree core}

For a real square matrix $A$, we denote the spectral radius of $A$ by $\rho(A).$ For non negative square matrices $A$ and $B$ with order of $B$ is greater or equal to order of $A$ , by the notation $A\ll B$ we mean that there exists permutation matrices $P$ and $Q$ such that $P^TAP$ is entry wise dominated by a principal submatrix of $Q^TBQ$, with strict inequality in at least one place, in case $A$ and $B$ have same order. A useful fact from the Perron-Frobenius theory is that if $B$ is irreducible and $A\ll B$ then $\rho(A)<\rho(B).$ We will now prove the main result of this section.

\begin{theorem}\label{th-01}
Among all trees on $n\geq 5$ vertices the distance between the subtree core and the characteristic set is maximized by some path-star tree.
\end{theorem}
\begin{proof}
Let $T$ be a tree on $n$ vertices. Our aim is to construct a path-star tree $P_{n-g,g}$ such that $d_{P_{n-g,g}}(S_c,\chi)\geq d_T(S_c,\chi).$ By Lemma \ref{L-03}, $d_{P_{n-2,2}}(S_c,\chi)=0,$ so we assume that $d_T(S_c,\chi)\geq 1.$

As the subtree core of a tree consists of either a vertex or two adjacent vertices and the characteristic set of a tree consists of a vertex or an edge (two adjacent vertices), so we need to consider four cases. Here, we prove the case when subtree core consists of two adjacent vertices and the characteristic set consists of an edge. The proofs of the other cases are similar.

Let $\chi(T)=\{u_1,v_1\}$  and $S_c(T)=\{u_2,v_2\}$. Also suppose that $d_T(S_c,\chi)= d(v_1,u_2).$ Let $C_1,C_2,\ldots,C_k$ be the components of $T - v_2$ where $C_1$ is the component containing $u_2.$   If $|V(C_i)|=1$ for $2\leq i\leq k$  then rename the tree $T$ as $\tilde{T}.$ Otherwise,  let $C\equiv\cup_{i=2}^kC_i$ and $|V(C)|=s$. Construct a new tree $\tilde{T}$ from $T$ by removing $C$ and adding $s$ pendant vertices at $v_2.$ By Lemma \ref{SC-01}, $S_c(\tilde{T})=\{v_2\}.$ we will now check the effect of this perturbation on the distance between characteristic set and subtree core in $\tilde{T}$. For that we will obtain $\tilde{T}$ from $T$ little differently. In $T$, at $v_1$ let $T_1$ be the component containing $u_1$ and $T_2$ be the component containing $u_2$. By Theorem \ref{pt1}, $T_1$ is the only Perron component at $v_1$ in $T.$ In $T$ at $v_1$, replace the component  $T_2$ by another component $\tilde{T_2}$, where  $\tilde{T_2}$ is obtained from $T$ by removing $C$ and adding $s$ pendant vertices at $v_2.$  The new tree is $\tilde{T}$ and by Lemma \ref{L-01}, $\hat{L}(\tilde{T_2})^{-1} \ll \hat{L}(T_2)^{-1} .$ So, $\rho(\hat{L}(\tilde{T_2})^{-1}) <\rho(\hat{L}(T_2)^{-1}) $ and hence in $\tilde{T}$ at $v_1$,  $T_1$ is the only Perron component. By Theorem \ref{pt3},  the characteristic set of $\tilde{T}$ is either $\{u_1,v_1\}$ or moves away from $v_2.$ So $d_{\tilde{T}}(S_c,\chi)\geq d_T(S_c,\chi).$

If $\tilde{T}$ is a path-star tree, then the result follows. Suppose $\tilde{T}$ is not a path-star tree. In $\tilde{T}$, let $v_3$ be the vertex which is either the characteristic vertex or an end point of the characteristic edge $\{u_3,v_3\}$ which is nearer to the subtree core $v_2.$  Let $A_1,A_2,\ldots,A_p$ be the connected components of $\tilde{T} - v_3$ with $A_1$ as the component containing the subtree core. If $p=2$ and $A_2$ is a path then rename the tree $\tilde{T}$ as $\hat{T}.$ Otherwise, let $\tilde{C}=\cup_{i=2}^p A_i$. Construct a new tree $\hat{T}$ from $\tilde{T}$ by replacing $\tilde{C}$ with a path $P$ on $|\tilde{C}|$ vertices. Then by Lemma \ref{L-01}, $\hat{L}(\tilde{C})^{-1} \ll \hat{L}(P)^{-1} .$  So, $\rho(\hat{L}(\tilde{C})^{-1}) <\rho(\hat{L}(P)^{-1}) $ and hence in $\hat{T}$ at $v_3$,  the component not containing $v_2$ is the only Perron component. By Theorem \ref{pt3},  the characteristic set of $\hat{T}$ is either $\{u_3,v_3\}$ or moves away from $v_2.$ The tree $\hat{T}$ can also be obtained from $\tilde{T}$ by following the perturbation mentioned in Lemma \ref{SC-02}. Hence by Lemma \ref{SC-02}, $S_c(\hat{T})=\{v_2\}.$ So $d_{\hat{T}}(S_c,\chi)\geq d_{\tilde{T}}(S_c,\chi).$

If $\hat{T}$ is a path-star tree, then the result follows. Otherwise, $\hat{T}$ has three parts. The first part is the path from vertex $1$ to $u_3$, second is a tree $T'$ containing $v_3$ and $u_2$ and the third part is the star $K_{1,s}$ centred at $v_2.$ Clearly $T'$ is not a path. Let $Q$ be the $v_3-u_2$ path in $T'$ and let there are $l$ vertices of $T'$ which are not in $Q.$ Delete all the vertices from $T'$ which are not in the path $Q$ and add a path on  $l$ vertices at $1$ to form a new tree $\bar{T}$. At $v_3$ in $\bar{T}$ there are two components and the component containing $u_3$ is the Perron component. By Theorem \ref{pt3},  the characteristic set of $\bar{T}$ is either $\{u_3,v_3\}$ or moves away from $v_2.$ The tree $\bar{T}$ can also be obtained from $\hat{T}$ by following the perturbation mentioned in Lemma \ref{SC-02}. Hence by Lemma \ref{SC-02}, $S_c(\hat{T})=\{v_2\}.$ So $d_{\bar{T}}(S_c,\chi)\geq d_{\hat{T}}(S_c,\chi).$ This proves the result.
 \end{proof}
 
For $1\leq i\leq n-g,$ we have
\begin{equation}\label{eq20}
f_{P_{n - g, g}}(i) = i(n-g-i) + i(2^g).
\end{equation}
Here the first term denotes the number of subtrees of $P_{n-g,g}$ containing the vertex $i$ but not $n-g$, while the second term counts the number of subtrees of $P_{n-g,g}$ containing both $i$ and $n-g$.

 \begin{lemma}\label{th-02}
 Let $g_0$ be the smallest positive integer such that $2^{g_0}+1>n-g_0.$ Then for $1 \leq k \leq g_{0}-2$,  a vertex   $n-g_0-\alpha\in S_c(P_{n-g_0+k,g_0-k})$ for some $\alpha\geq 0.$
 \end{lemma}
 
 \begin{proof}
 Since $2^{g_0}+1>n-g_0,$ by Lemma \ref{SC-03}, $S_c(P_{n-g_0,g_0})={n-g_0}.$ For $1 \leq k \leq g_0-2,$ by (\ref{eq20}), we have
 $$f_{P_{n-g_0+k,g_0-k}}(n-g_0)=(n-g_0)(k+2^{g_0-k})$$ and 
 $$f_{P_{n-g_0+k,g_0-k}}(n-g_0+1)=(n-g_0+1)(k-1+2^{g_0-k}).$$ Then 
 \begin{align*}
 &f_{P_{n-g_0+k,g_0-k}}(n-g_0)-f_{P_{n-g_0+k,g_0-k}}(n-g_0+1)\\&=n-g_0-(k-1+2^{g_0-k})\\&=[n-(g_0-1)-(2^{g_0-1}+1)]+2^{g_0-1}-2^{g_0-k}-k+1\\&\geq 2^{g_0-k}(2^{k-1}-1)-k+1\\&\geq 0.
 \end{align*}
 By Lemma \ref{SC-00}, the function $f_T$ is strictly concave, hence the result follows.
 \end{proof}
 
 \begin{theorem}\label{th-0}
 Let $g_0$ be the smallest positive integer such that $2^{g_0}+1>n-g_0.$ Then among all trees on $n\geq 5$ vertices the path-star tree $P_{n-g_0,g_0}$ maximizes the distance between the subtree core and the characteristic set.
\end{theorem}

\begin{proof}
 By Theorem \ref{th-01}, we need to consider  path-star trees only. Consider the path-star tree $P_{n-g_0,g_0}$. Let $d_{P_{n-g_0,g_0}}(\chi,S_c)=r.$ We show that for $g\neq g_0$, $d_{P_{n-g,g}}(\chi,S_c)\leq r.$ By Lemma \ref{SC-03}, $S_c(P_{n-g_0,g_0})=\{n-g_0\}.$ So $\chi(P_{n-g_0,g_0})=\{n-g_0-r\}$ or $\{n-g_0-r-1,n-g_0-r\}.$

First suppose that $g=g_0+1$. Then by Lemma \ref{SC-03}, $S_c(P_{n-g,g})=S_c(P_{n-g_0-1,g_0+1})=\{n-g_0-1\}$. By Lemma \ref{CS-02},  if $\chi(P_{n-g_0,g_0})=\{n-g_0-r\}$, then
	$\chi(P_{n-g_0-1,g_0+1})=\{n-g_0-r-1,n-g_0-r\} $ and  if
	 $\chi(P_{n-g_0,g_0})=\{n-g_0-r-1,n-g_0-r\}$, then $\chi(P_{n-g_0-1,g_0+1})= \{n-g_0-r-2,n-g_0-r-1\}$  or $\{n-g_0-r-1\} $ or $\{n-g_0-r-1,n-g_0-r\}$. It is easy to check that $d_{P_{n-g_0-1,g_0+1}}(\chi,S_c) \leq r.$ Same argument holds for any $g > g_0.$

Let $1\leq k\leq g_0-2.$ Now suppose that $g=g_0-k$. Then by Lemma \ref{th-02}, the subtree core of $P_{n-g_0+k,g_0-k}$ moves at least $k$ steps and by Lemma \ref{CS-02} the characteristic set moves at most $k$ steps towards center. So $d_{P_{n-g_0+k,g_0-k}}(\chi,S_c) \leq r$ and hence the result follows.
 \end{proof}
 	
We define $\delta_n(\chi,S_c)=max\{d_T(\chi,S_c): $T is a tree on $n$ vertices$\}$. Analogously we define  $\delta_n(C,S_c)$, $\delta_n(C,C_d)$, $\delta_n(C_d,S_c)$, $\delta_n(C,\chi)$ and $\delta_n(C_d,\chi)$. To get the value of  $\delta_n(\chi,S_c)$, it is important to know the position of the characteristic set of a path-star tree. In this regard,  Abreu et.al.  proved a result (see \cite{Afjk}, Lemma 2.1) which gives the Perron value of a path-star branch at  a vertex $v$ of $T.$ We propose the following conjecture related to the characteristic set of path-star trees.\\

\textbf{Conjecture:} The characteristic set of a path-star tree contains an edge.\\

In \cite{Afjk}, the authors have established the values  of $\lim_{n\rightarrow \infty}\frac{\delta_n(C,\chi)}{n}$ and $\lim_{n\rightarrow \infty}\frac{\delta_n(C_d,\chi)}{n}$.  We will now do the same for other four remaining  such maximum distances.

\begin{theorem}(\cite{P},Theorem 3.4 and 3.5)\label{CC}
Among all trees on $n\geq 5$ vertices, the path-star tree $P_{n-\lfloor \frac{n}{2}\rfloor, \lfloor \frac{n}{2}\rfloor}$  maximizes the distance between center and centroid. Also $\delta_n(C,C_d)=\lfloor\frac{n-3}{4}\rfloor.$
\end{theorem}

\begin{corollary}
$ \lim_{n \rightarrow \infty} \frac{\delta_n(C,C_d)}{n}=\frac{1}{4}.$
\end{corollary}
\begin{proof}
Follows from Theorem \ref{CC}.
\end{proof}

\begin{theorem}(\cite{Dp},Proposition 2.7 and Corollary 2.8)\label{CSc}
Let $g_0$ be the smallest positive integer such that $2^{g_0}+1>n-g_0.$ Then among all trees on $n\geq 5$ vertices, the path-star tree $P_{n-g_0,g_0}$ maximizes the distance between center and subtree core. Also we have $\delta_n(C,S_c)=\lfloor\frac{n-g_0}{2}\rfloor -1.$
\end{theorem}

\begin{corollary}\label{cor1}
$ \lim_{n \rightarrow \infty} \frac{\delta_n(C,S_c)}{n}=\frac{1}{2}.$
\end{corollary}
\begin{proof}
Since   $g_0$ is the smallest positive integer such that $2^{g_0}+g_0> n-1,$ so $2^{g_0-1}+g_0-1\leq n-1.$  This implies $2^{g_0-1}< n.$ Taking logarithm with base 2 on both side , we have $g_0< 1+\log_2 n.$ As $n\geq 5$, So $0< g_0< 1+\log_2 n.$ Since $ \lim_{n \rightarrow \infty} \frac{\log_2 n}{n}=0$, so $ \lim_{n \rightarrow \infty} \frac{g_0}{n}=0.$

Hence $ \lim_{n \rightarrow \infty} \frac{\delta_n(C,S_c)}{n}=\lim_{n \rightarrow \infty}\frac{\lfloor\frac{n-g_0}{2}\rfloor -1}{n}=\frac{1}{2}$
\end{proof}

 \begin{theorem}(\cite{Dp},Proposition 3.5 and Theorem 3.6)\label{CdSc}
 Let $g_0$ be the smallest positive integer such that $2^{g_0}+1>n-g_0.$ Then among all trees on $n\geq 5$ vertices, the path-star tree $P_{n-g_0,g_0}$ maximizes the distance between centroid and subtree core. Also we have $\delta_n(C_d,S_c)=\lfloor\frac{n-1}{2}\rfloor -g_0.$
 \end{theorem}
 
 \begin{corollary}\label{cor2}
$ \lim_{n \rightarrow \infty} \frac{\delta_n(C_d,S_c)}{n}=\frac{1}{2}.$
\end{corollary}
\begin{proof}
Since $ \lim_{n \rightarrow \infty} \frac{g_0}{n}=0$, The result follows from  Theorem \ref{CdSc}.
\end{proof}

 \begin{proposition}(\cite{P},Theorem 3.3 and \cite{Dp}, Proposition 4.1)\label{SC-05}
 In any path-star tree the following hold.
 \begin{enumerate}
 \item The characteristic set lies in the path connecting the center and the centroid.
 \item The centroid lies in the path connecting the center and the subtree core.
 \end{enumerate}
 \end{proposition}
 
\begin{theorem}
 $ \lim_{n \rightarrow \infty} \frac{\delta_n(\chi,S_c)}{n}=\frac{1}{2}.$
\end{theorem}
 \begin{proof}
By Theorem \ref{th-0}, we have $\delta_n(\chi,S_c)=d_{P_{n-g_0,g_0}}(\chi,S_c)$. Also by Theorem \ref{CSc}, $\delta_n(C,S_c)=d_{P_{n-g_0,g_0}}(C,S_c)$ and by Theorem \ref{CdSc} $\delta_n(C_d,S_c)=d_{P_{n-g_0,g_0}}(C_d,S_c)$. Now from Proposition \ref{SC-05}, it follows that
$$ \delta_n(C_d,S_c)\leq \delta_n(\chi,S_c)\leq \delta_n(C,S_c)$$
 
$$ \Rightarrow \lim_{n \rightarrow \infty} \frac{\delta_n(C_d,S_c)}{n}\leq  \lim_{n \rightarrow \infty} \frac{\delta_n(\chi,S_c)}{n}\leq  \lim_{n \rightarrow \infty} \frac{\delta_n(C,S_c)}{n}.$$

 By Corollary \ref{cor1} and Corollary \ref{cor2}, $ \lim_{n \rightarrow \infty} \frac{\delta_n(\chi,S_c)}{n}=\frac{1}{2}.$
 
 \end{proof}
 
 \section{Trees with fixed diameter}
  
In this section, we will try to extremize the pairwise distance between  different central parts of trees on $n$ vertices with diameter $k$. If $k=1$ then $n=2$ and $K_2$ is the only such tree. If $k=2$ then $n\geq 3$ and star is the only such tree. So we can consider $3\leq k\leq n-1.$ For a tree $T$ and the edge $e=\{u,v\} \in E(T)$, let $T_e(u)$ denotes the component of $T - e$ containing $u$. The following result will be helpfull in this regard.
 
\begin{proposition} \label{SC-04}(\cite{Sswy},Proposition 1.7)
Let $T$ be a tree. A vertex $u \in S_c(T)$ if and only if for each neighbour $v$ of $u$, $f_{T_e(u)}(u)\geq f_{T_e(v)}(v)$ where $e=\{u,v\}$. Furthermore if $u \in S_c(T)$ and equality holds then $v \in S_c(T)$.
 \end{proposition}

We denote the set of all trees on $n$ vertices with diameter $k$ by $\Gamma_n^k$.Take the path $P_{k+1}=[v_1,v_2,\ldots,v_{k+1}]$ on $k+1$ vertices.  Construct  a new tree from $P_{k+1}$ by adding $n-k-1$ pendant vertices at the vertex $v_{\lfloor\frac{k+2}{2}\rfloor}$. We dnote the new tree by $T_{n,k}$. Clearly $T_{n,k}\in \Gamma_n^k$.  
\begin{lemma}
$d_{T_{n,k}}(C,C_d)=d_{T_{n,k}}(C,S_c)=d_{T_{n,k}}(C_d,S_c)=d_{T_{n,k}}(C,\chi)=d_{T_{n,k}}(C_d,\chi)=d_{T_{n,k}}(\chi,S_c)=0.$
\end{lemma}
\begin{proof}
We consider two cases depending on $k$ is even or odd.\\
\noindent\textbf{Case I:} $k$ is even\\
It is easy to check that $C(T_{n,k})=C_d(T_{n,k})=\{v_{\frac{k+2}{2}}\}$. Also at $v_{\frac{k+2}{2}}$ in $T_{n,k}$ there are  two Perron components (since $k>2$), so by Theorem \ref{pt2} $\chi(T_{n,k})=\{v_{\frac{k+2}{2}}\}.$   

The subtree core $S_c(T_{n,k})$ does not contain any pendent vertex (see \cite{Dp}, Remark 1.5). Consider the edge $e=\{v_{\frac{k}{2}},v_{\frac{k+2}{2}}\}$. Let $C_1$ and $C_2$ be the components of $T - e$ containing  $v_{\frac{k}{2}}$ and $v_{\frac{k+2}{2}}$, respectively. Then $f_{C_2}(v_{\frac{k+2}{2}})> f_{C_1}(v_{\frac{k}{2}}).$ By symmetry and Proposition \ref{SC-04}, we have $S_c(T_{n,k})=\{v_{\frac{k+2}{2}}\}$ and hence the result follows.

\noindent\textbf{Case II:} $k$ is odd\\
If $n=k+1$, then $T_{n,k}$ is a path and $C(T_{n,k})=C_d(T_{n,k})=\chi(T_{n,k})=S_c(T_{n,k})=\{v_{\frac{k+1}{2}},v_{\frac{k+3}{2}}\}.$ 

If $n>k+1$ then $C(T_{n,k})=\{v_{\frac{k+1}{2}},v_{\frac{k+3}{2}}\}$ and $C_d(T_{n,k})=\{v_{\frac{k+1}{2}}\}$. At $v_{\frac{k+1}{2}}$, the component containing $v_{\frac{k+3}{2}}$ is the only Perron component and at $v_{\frac{k+3}{2}}$ the component containing $v_{\frac{k+1}{2}}$ is the only Perron component, so by Theorem \ref{pt1}, $\chi(T_{n,k})=\{v_{\frac{k+1}{2}},v_{\frac{k+3}{2}}\}.$ Also using similar technique as in Case  I, it can be checked that $S_c(T_{n,k})=\{v_{\frac{k+1}{2}}\}.$ Hence the result follows.

\end{proof}

\begin{corollary}
Among all trees on $n$ vertices and diameter $k$, the minimum distance between any two central parts is $0.$
\end{corollary}
  
We will now try to maximize the distance between any two central parts over $\Gamma_n^k$.  The same is studied between the central parts center and centroid and the central parts center and subtree core (see \cite{Sswy}, Proposition 4.1 and Proposition 4.2). Next we will discuss about the central parts center and characteristic set.
\begin{theorem}
 The path-star tree $P_{k,n-k}$ maximizes the distance between the center and the characteristic set over $\Gamma_n^k$.
 \end{theorem}
\begin{proof}
Let $T\in \Gamma_n^k$.  We will prove that $d_{P_{k,n-k}}(C,\chi)\geq d_T(C,\chi)$. Without loss of generality we can take $d_T(C,\chi)\geq 1.$ The center of $T$ lies in all the longest paths of $T$. We consider two cases depending on the position of characteristic set of $T$.\\
\noindent\textbf{Case I:} Characteristic set of $T$ lies in a longest path\\
Let $P$ be  a  longest path of $T$  containing both $C(T)$ and $\chi(T)$. Then the diameter of the path $P$ is $k.$ Let $v$ be the vertex in the characteristic set which is farthest from $C(T)$. Let $C_1,C_2,\ldots,C_l$ be the components of $T - v$ where $C_1$ is the component containing the center of $T.$ Let $C\equiv \cup_{j=2}^lC_j$ and let the number of vertices in $C$ be $s.$ Construct a new tree $\tilde{T}$ from $T$ by replacing $C$  with a path-star tree $P_{g,s-g}$ at $v$, where $g=\max\{diam(C_2),\cdots,diam(C_l)\}$. Then $\tilde{T}\in \Gamma_n^k.$ Suppose $\tilde{M}$ is the bottleneck matrix of $P_{g,s-g}$ at $v$ in $\tilde{T}$. Then by Lemma \ref{L-01} $\tilde{M}\gg \hat{L}(C)^{-1}$ and the characteristic set of $\tilde{T}$ is either same as characteristic set of $T$ or it  moves away from its center towards the path-star part. So, $d_{\tilde{T}}(C,\chi)\geq d_T(C,\chi).$

If $\tilde{T}$ is a path-star tree then the result follows. Otherwise at $v$, one of the components in $\tilde{T}$ is a path-star tree. In the other component at $v$, choose the longest path $P_1$ which contains the center of $\tilde{T}$. Delete the vertices which are not on $P_1$, and add the same number of vertices (as pendants) to the star part (of the other component) to get a new tree $\hat{T}.$  Clearly $\hat{T}$ is the path-star tree $P_{k,n-k}$. Then $C(\tilde{T})= C(\hat{T})$ and the characteristic set of $\hat{T}$ is either same as characteristic set of $\tilde{T}$ or it  moves away from its center towards the path-star part. So, $d_{\hat{T}}(C,\chi)\geq d_{\tilde{T}}(C,\chi)\geq d_T(C,\chi).$ Hence the result follows.

\noindent\textbf{Case II:} Characteristic set of $T$ does not lie in any of the  longest path\\
Let $P$ be  the longest path of $T$  containing both $C(T)$ and $\chi(T)$. Then the diameter of the path $P$ is less than $k.$ Let $v$ be the vertex in the characteristic set which is farthest from  $C(T)$ and let $u$ be the pendant vertex of $P$ farthest from $v.$ Let $d_T(u,v)=\alpha.$ Let $C_1,C_2,\ldots,C_l$ be the components of $T - v$ where $C_1$ is the component containing  $C(T).$  Since $diam(P)<k$ and $d(u,v)=\alpha$, so $\max\{diam(C_2),\cdots,diam(C_l)\}\leq k-\alpha-2$. Let $C\equiv \cup_{j=2}^lC_j$ and let the number of vertices in $C$ be $s.$ As $\alpha>\frac{k}{2}$ and  $v$ is in the characteristic set, so $s> k-\alpha.$ Construct a new tree $\tilde{T}$ from $T$ by replacing $C$  with a path-star tree $P_{k-\alpha-1,s-(k-\alpha-1)}$ at $v$. Then $\tilde{T}\in \Gamma_n^k.$ Suppose $\tilde{M}$ is the bottleneck matrix of $P_{k-\alpha-1,s-(k-\alpha-1)}$ at $v$ in $\tilde{T}$. Then by Lemma \ref{L-01} $\tilde{M}\gg \hat{L}(C)^{-1}$ and the characteristic set of $\tilde{T}$ is either same as characteristic set of $T$ or it  moves away from its center towards the path-star part. So, $d_{\tilde{T}}(C,\chi)\geq d_T(C,\chi).$ 

Now the center and characteristic set of $\tilde{T}$ lies in a longest path of it and the result follows from Case I.
\end{proof}

 For positive integers $l,m,k$ with $n=l+m+k$, let $T(l,m,k)$ be the tree of order $n$ obtained by taking the path  $P_k:v_1v_2\cdots v_k$  and adding $l$ pendant vertices adjacent to $v_1$ and $m$ pendant vertices adjacent to $v_k$. Note that $T(l,m,k)\in \Gamma_n^{k+1}.$
 
\begin{theorem}
 Let $k\leq \lceil\frac{n}{2}\rceil$. Then over $\Gamma_n^k$, the distance between the centroid and the characteristic set is maximized by the tree  $T(n-\lfloor\frac{n}{2}\rfloor-k+1,\lfloor\frac{n}{2}\rfloor,k-1)$.
 \end{theorem}
\begin{proof}
Let $T\in \Gamma_n^k$. Without loss of generality we can take $d_T(C_d,\chi)\geq 1.$ Let  $u \in \chi(T)$ and $v \in C_d(T)$ such that $d_T(C_d,\chi)=d_T(u,v)$. Let $C_1,C_2,\ldots,C_l$ be the components of $T\setminus v$ where $C_1$ is the component containing $\chi(T).$ Let $C\equiv \cup_{j=2}^lC_j$ and let the number of vertices in $C$ be $s.$ Construct $\tilde{T}$ from $T$ by removing $C$ and adding $s$ pendant vertices at $v.$ Observe that  $C_d(\tilde{T})=\{v\}$  and $diam(\tilde{T})\leq k.$  Let $M$ be the bottleneck matrix of the component of $T - u$ containing $v$ and let $\tilde{M}$ be the bottleneck matrix of the component of $\tilde{T}\setminus u$ containing $v$. Then by Lemma \ref{L-01} $M\gg \tilde{M}$ and by Theorem \ref{pt3}  the characteristic set of $\tilde{T}$ is either same as characteristic set of $T$ or it  moves away from $v.$ So, $d_{\tilde{T}}(C_d,\chi)\geq d_T(C_d,\chi).$

Let $w\in \chi(\tilde{T})$ such that $w$ is nearest to $v$. Let $D_1,D_2,\ldots,D_p$ be the components of $\tilde{T} - w$ where $D_1$ is the component containing the vertex $v.$ Let $D\equiv \cup_{j=2}^pC_j$ and let the number of vertices in $D$ be $q.$ Construct a new tree $\hat{T}$ from $\tilde{T}$ by replacing $D$ at $w$  with a path-star tree $P_{g,q-g}$ at $w$, where $g=\max\{diam(C_2),\cdots,diam(C_l)\}$. Then   the characteristic set of $\hat{T}$ is either same as characteristic set of $\tilde{T}$ or it  moves away from $v.$ So, $d_{\hat{T}}(C_d,\chi)\geq d_{\tilde{T}}(C_d,\chi).$ Also $diam(\hat{T})=diam(\tilde{T})\leq k.$

Let $w_1$ be the center of the star part of $P_{g,q-g}$. In $\hat{T}$, let $v'$ and $w_1'$ be the non pendant vertices adjacent to $v$ and $w_1$, respectively. Consider the maximal subtree of $\hat{T}$ not containing both $v$ and $w.$ Delete all the vertices of the subtree which are not in the $w_1'-v'$ path and add them as a pendant vertex at $w_1$ to form a new tree $\hat{T_1}$ from $\hat{T}.$ Clearly $d_{\hat{T_1}}(C_d,\chi)\geq d_{\tilde{T}}(C_d,\chi)$ and $diam(\hat{T_1})=diam(\hat{T})\leq k.$

Since  $C_d(\hat{T_1})=\{v\}$, so at least $\lfloor\frac{n}{2}\rfloor$ pendant verices are adjacent to $v.$ If the number of pendant vertices adjacent to $v$ is $\alpha$ then remove $\alpha - \lfloor\frac{n}{2}\rfloor$ pendant vertices from $v$ and add them as  pendant vertices at $w_1$ to form a new tree $\hat{T_2}$ from $\hat{T_1}$. Then $C_d(\hat{T_2})=C_d(\hat{T_1})=\{v\}$ and either $\chi(\hat{T_2})=\chi(\hat{T_1})$ or $\chi(\hat{T_2})$ moves towards $w_1.$ So $d_{\hat{T_2}}(C_d,\chi)\geq d_{\hat{T_1}}(C_d,\chi)$ and $diam(\hat{T_2})=diam(\hat{T_1})\leq k.$

If $diam(\hat{T_2})=k$, then we are done. Otherwise let $u'\in \chi(\hat{T_2})$ such that $u'$ is closer to $v$ and $d_{\hat{T_2}}(u',v)=\beta.$ Let $E_1$ be the component of $\hat{T_2}$ at $u'$ containing $w_1$. Then $diam(E_1)<k-\beta -2=\gamma_1$(say) and order of $E_1$ is $n-(\lfloor \frac{n}{2} \rfloor+\beta +1)=\gamma_2$(say). Construct a new tree $\hat{T_3}$ from $\hat{T_2}$ by replacing $E_1$ at $u'$ with a path-star tree $P_{\gamma_1,\gamma_2-\gamma_1}$. The new tree $\hat{T_3}$ is the tree $T(n-\lfloor\frac{n}{2}\rfloor-k+1,\lfloor\frac{n}{2}\rfloor,k-1)$ and $d_{\hat{T_3}}(C_d,\chi)\geq d_{\hat{T_2}}(C_d,\chi)$. Hence the result follows.
\end{proof}
It will be nice to find a tree which maximizes the distance between centroid and characteristic set over $\Gamma_n^k$ for $k> \lceil\frac{n}{2}\rceil.$ Also it seems difficult to find the trees which maximize the distance between centroid and subtree core and the distance between characteristic set and subtree core over $\Gamma_n^k$.

\noindent{\bf Addresses}:\\

\noindent 1) School of Mathematical Sciences,\\
National Institute of Science Education and Research (NISER), Bhubaneswar,\\
P.O.- Jatni, District- Khurda, Odisha - 752050, India\medskip

\noindent 2) Homi Bhabha National Institute (HBNI),\\
Training School Complex, Anushakti Nagar,\\
Mumbai - 400094, India\medskip

\noindent E-mails: dinesh.pandey@niser.ac.in, klpatra@niser.ac.in

\end{document}